\documentclass[reqno,9pt]{amsart}

\usepackage{amsthm, amsmath, amssymb, amsfonts, graphicx}

\usepackage[OT4]{fontenc}
\usepackage[utf8]{inputenc}
\usepackage[english]{babel}

\usepackage{bbm}

\usepackage[inline]{enumitem}

\usepackage{float}
\restylefloat{table}

\usepackage[colorlinks=true, pdfstartview=FitV, linkcolor=blue,
            citecolor=blue, urlcolor=blue]{hyperref}
\usepackage[usenames]{color} 

\usepackage[round]{natbib}

\usepackage{graphicx}
\usepackage[font=sl,labelfont=bf,width=\textwidth]{caption}

\usepackage{color}

\DeclareMathOperator{\cov}{Cov}
\DeclareMathOperator{\cor}{Cor}

\DeclareMathOperator{\var}{Var}

\def\bP{\mathbb{P}}
 
\usepackage{multirow}
\usepackage{amsthm}
\newtheorem{theorem}{Theorem}

\theoremstyle{definition}

\newtheorem{example}[theorem]{Example}
\theoremstyle{remark}
\newtheorem{remark}[theorem]{Remark}

\DeclareMathOperator{\Cov}{Cov} 
\def\bP{\mathbb{P}}

\def\bR{\mathbb{R}}
\def\bN{\mathbb{N}}
\usepackage{bbm}

\usepackage{amssymb}

\usepackage{placeins} 
\usepackage{graphicx}

\usepackage{setspace}
\usepackage[foot]{amsaddr}

\begin{document}

\title[A note on the zero conditional correlation]{A note on the equivalence between the conditional uncorrelation and the independence of random variables}

\author{Piotr Jaworski$^{\ast}$}
\author{Damian Jelito$^{\dagger}$}
\author{Marcin Pitera$^{\dagger}$}
\address{$^{\ast}$Institute of Mathematics, University of Warsaw, Banacha 2, 02-097 Warsaw, Poland}
\address{$^{\dagger}$Institute of Mathematics, Jagiellonian University, S. {\L}ojasiewicza 6, 30-348 Krak{\'o}w, Poland}
\email{p.jaworski@mimuw.edu.pl, damian.jelito@uj.edu.pl, marcin.pitera@uj.edu.pl}

\maketitle

\vspace{-1cm}
\begin{abstract}
It is well known that while the independence of random variables implies zero correlation, the opposite is not true. Namely, uncorrelated random variables are not necessarily independent. In this note we show that the implication could be reversed if we consider the localised version of the correlation coefficient. More specifically, we show that if random variables are conditionally (locally) uncorrelated for any quantile conditioning sets, then they are independent. For simplicity, we focus on the absolutely continuous case. Also, we illustrate potential usefulness of the stated result using two simple examples.
\vspace{0.2cm}

\end{abstract}


\section{Introduction}
The concept of linear correlation was first presented in~\cite{Gal1889}, see~\cite{Sti1989} for a historical note on the correlation invention. While mathematically simple and elegant, statistical analysis based on correlation measurement could be confusing and lead to subtle errors if not treated with caution, see e.g. \cite{Ald1995}, \cite{Bel2015}, and references therein. Since correlation aims to measure the linear dependence between random variables, it often fails to properly capture non-linear structures. Although the dependence could be fully described using the copula function, it is more appealing, especially to practitioners, to use simpler (numeric) characteristics to describe the degree of dependence, see~\cite{Nel2006} or \cite{KotDro2001}. Because of that, a lot of alternative measures of dependence have been proposed in the literature and this field is constantly evolving. Let us alone mention the concepts of concordance measures, entropy correlations, projection correlations, tail correlations, partial and conditional correlations, maximal correlations, time-varying dynamic correlations, local Gaussian correlations, and distance correlations based on energy statistics, see~\cite{Sca1984,RaoSetXuCheTagPri2011,ZhuXuRunZho2017,AkeBruRobSimWei1984,BabShiSib2004,KenHuaVodHavSta2015,Wit1975,Aie2013,TjoHuf2013,SzeRizBak2007,TjoOtnSto2022}, and references therein.

Typically, it is expected that the zero value of a given dependence measure should, in some sense, imply independence. What is interesting, at first, the concept of null linear correlation was often mixed with independence and it took some time for statisticians to distinguish between null correlation and statistical independence, see \cite{Dav2009}. Of course, it is currently well known that while the independence of random variables implies zero correlation, the opposite is not true, see \cite{Bro1986} for a classroom example.

In this short paper we answer a simple question about how one can revert the aforementioned implication, i.e. whether one can use linear correlation to study (proper) independence. Allowing non-linear transforms of random variables, the reverse implication is in fact trivially true as one of the alternative definition of independence states that two random variables $X$ and $Y$ are independent, if $f(X)$ and $g(X)$ are uncorrelated for any test functions $f$ and $g$; in fact, it is sufficient to consider set indicator functions to directly recover the definition of independence. Still, this characterisation is not appealing from practical perspective since it is hard to pre-set the family of test functions that would work for any arbitrary pair of random variables and lead to efficient statistical setup. Another approach is to consider a localised version of correlation and study its properties, see Section 6 in \cite{KotDro2001} for details. In this paper, following \cite{JawPit2020}, we propose to bind those two approaches together and consider a family of conditional correlations, where the conditioning is based on the quantile set linked to the values of $X$ and $Y$, see Section~\ref{S:main} for details. In the main result of this paper, Theorem~\ref{th:main}, we show that null correlation on every quantile set implies independence of random variables so that the aforementioned implication could be reverted by looking locally into linear relation between random variables. Namely, we show that random variables are independent if and only if they are locally linearly independent. Due to our best knowledge, quite surprisingly, this result has not been stated nor analysed previously in the literature -- this is most likely due to the fact that localised correlations considered so far were not bound directly to quantile sets allowing efficient local treatment.

We believe that our proposal could be appealing to practitioners and could lead to development of new efficient statistical frameworks. In fact, the sample version of (local) quantile correlation could be easily computed using rank statistics and exhibits statistical properties similar to the unconditional correlation; this aspect is left to future research. In other words, the results presented in this paper lay the theoretical ground to expansion of the statistical framework based on quantile conditional moments which already proved to be useful, see e.g.~\cite{HebZimPitWyl2019},~\cite{JelPit2018}, and \cite{PitCheWyl2021}. As an example, one could define the conditional version of the auto-correlation function that could be used to study time-series which exhibits heavy tails, see Example~\ref{ex:2} for details, or study the tail-based correlations to recover dependence conditioned on tail-events, see~\cite{JawPit2015}.

This paper is organised as follows. In Section~\ref{S:prel}, we introduce the basic notation and define the concept of quantile conditional correlation. In Section~\ref{S:main}, we state and prove the main result, Theorem~\ref{th:main}, together with its multidimensional extensions. Finally, in Section~\ref{S:examp}, we show two simple examples that illustrate how our approach could be used to study dependence between random variables.

\section{Preliminaries}\label{S:prel}

Let $(\Omega, \mathcal{F}, \mathbb{P})$ be a probability space and let $(X,Y)$ be a random vector defined on this space. By Sklar's theorem, we know that the joint distribution of $(X,Y)$ can be represented as 
\begin{equation}\label{eq:copula}
\mathbb{P}[X\leq x, Y\leq y]=C(F_X(x), F_Y(y)), \quad x,y\in \mathbb{R},
\end{equation}
where $F_X$ and $F_Y$ denote the distributions of $X$ and $Y$, respectively, and $C$ is the copula function of the vector $(X,Y)$, see e.g. Theorem 2.3.3 in~\cite{Nel2006}. For simplicity, from now on we assume that the vector $(X,Y)$ is absolutely continuous and use $f_X$, $f_Y$, $f$, and $c$, to denote the density functions of $X$, $Y$, $(X,Y)$, and $C$, respectively. Also, we assume that $F_X$ and $F_Y$ are bijective, the vector $(X,Y)$ has a full (non-degenerate) support, and the copula density $c$ is continuous. In this case, the copula function $C$ is unique and can be easily recovered from the joint distribution using the formula $C(u,v)=\bP[X\leq Q_X(u),Y\leq Q_Y(v)]$, $u,v\in (0,1)$, where $Q_X:=F^{-1}_X$ and $Q_Y:=F^{-1}_Y$ are the quantile functions of $X$ and $Y$, respectively.

Now, let us introduce a notation associated with quantile conditional covariances.  Given a set $A\in\mathcal{F}$, we define the conditional covariance of $(X,Y)$ on $A$ by setting
\begin{equation}\label{eq:covariance}
\Cov_{A}[X,Y]:=\mathbb{E}\left[XY| A\right]-\mathbb{E}\left[X| A\right]\mathbb{E}\left[Y| A\right],
\end{equation}
provided that the expectations are well-defined.
In this paper, we are interested in quantile-based conditioning. Namely, given a vector $(X,Y)$ and {\it quantile splits} $0<p_1<q_1<1$ and $0<p_2<q_2<1$, we define the corresponding {\it quantile set} as
\begin{equation}\label{eq:A}
A:=\{Q_X(p_1)\leq X\leq Q_X(q_1)\}\cap \{Q_Y(p_2)\leq Y\leq Q_Y(q_2)\}\in \mathcal{F}.
\end{equation}
Note that since we assumed a full support, we get $\bP[A]>0$ for any quantile split. Also, since both $X$ and $Y$ are bounded on $A$, we get that \eqref{eq:covariance} is well-defined and finite. Thus, we can also define the corresponding conditional correlation by setting
\[
\cor_{A}[X,Y]:=\frac{\Cov_{A}[X,Y]}{\sqrt{\var_{A}[X]\var_{A}[Y]}},
\]
where $\var_A[X]:=\mathbb{E}[X^2|A]-\mathbb{E}^2[X|A]$ and $\var_A[Y]:=\mathbb{E}[Y^2|A]-\mathbb{E}^2[Y|A]$ are conditional covariances of $X$ and $Y$, respectively. 

From now on we assume that we are given specific quantile splits $0<p_1<q_1<1$ and $0<p_2<q_2<1$, and use $A$ to denote the corresponding quantile set as defined in \eqref{eq:A}. For brevity, we also introduce the corresponding value projection set $\tilde A:=[Q_X(p_1), Q_X(q_1)]\times [Q_Y(p_2),Q_Y(q_2)] \subset \mathbb{R}^2$.  
With this notation, we get that~\eqref{eq:covariance} could be expressed as
\begin{equation}\label{eq:covariance_dens}
\Cov_{A}[X,Y]=\frac{1}{\bP[A]}\int_{\tilde A}xyf(x,y)dydx -\frac{1}{\bP^2[A]}\int_{\tilde A}xf(x,y)dydx \int_{\tilde A}yf(x,y)dydx .
\end{equation}
We say that $X$ and $Y$ are conditionally uncorrelated on $A$, if $\cor_{A}[X,Y]=0$.

\section{Main result}\label{S:main}
In this section we present the main result of this note, which shows that the independence of random variables could be linked to their conditional uncorrelation on any quantile set.
\begin{theorem}\label{th:main}
Random variables $X$ and $Y$ are independent if and only if they are conditionally uncorrelated on every quantile set, i.e. for any quantile splits $0<p_1<q_1<1$ and $0<p_2<q_2<1$ and the related set $A$, we get $\cor_{A}[X,Y]=0$.
\end{theorem}

\begin{proof}
The fact that the independence of $X$ and $Y$ implies~$\cor_{A}[X,Y]=0$ for any $A\in\mathcal{F}$ follows from the standard argument which is omitted for brevity. Let us now assume that for any quantile splits $0<p_1<q_1<1$ and $0<p_2<q_2<1$ and the related set $A$ we get $\cor_{A}[X,Y]=0$ or equivalently
\begin{equation}\label{p:pr:equation:2}
\cov_{A}[X,Y]=0.
\end{equation}
First, let us show that for any $u_1, v_1, u_2, v_2\in (0,1)$ we have
\begin{equation}\label{eq:copula_dens}
c(u_1,v_1)c(u_2,v_2)-c(u_1,v_2)c(u_2,v_1)=0.
\end{equation} 
We start with deriving a useful representation of $\cov_{A}[X,Y]$ based on~\eqref{eq:covariance_dens}. For any set $A$ defined in~\eqref{eq:A}, using the fact that $f(x,y)=c(F_X(x),F_Y(y))f_X(x)f_Y(y)$, $x,y\in\bR$, and substituting $x=Q_X(u)$ and $y=Q_Y(v)$, we get
\begin{align}
\Cov_{A}[X,Y]&=\frac{1}{\bP[A]}\int_{p_1}^{q_1}\int_{p_2}^{q_2}Q_X(u)Q_Y(v)c(u,v)dv du  -\frac{1}{\bP^2[A]}\int_{p_1}^{q_1}\int_{p_2}^{q_2}Q_X(u)c(u,v)dv du \int_{p_1}^{q_1}\int_{p_2}^{q_2}Q_Y(v)c(u,v)dv du \nonumber\\
& = \frac{1}{\bP^2[A]}\int_{p_1}^{q_1}\int_{p_2}^{q_2}\int_{p_1}^{q_1}\int_{p_2}^{q_2}Q_X(u_1)Q_Y(v_1)\left(c(u_1,v_1)c(u_2,v_2)-c(u_1,v_2)c(u_2,v_1)\right)dv_2 du_2 dv_1 du_1\label{eq:th:main:new1}.
\end{align}
Let us define $H(u_1,v_1,u_2,v_2):=(Q_X(u_1)-Q_X(u_2))(Q_Y(v_1)-Q_Y(v_2))V_c(u_1,v_1,u_2,v_2)$, $u_1,v_1,u_2,v_2\in (0,1)$, with $V_c(u_1,v_1,u_2,v_2):=c(u_1,v_1)c(u_2,v_2)-c(u_1,v_2)c(u_2,v_1)$, $u_1,v_1,u_2,v_2\in (0,1)$, and note that $V_c$ is anti-symmetric in $(u_1,u_2)$ and $(v_1,v_2)$, i.e. we get $V_c(u_2,v_1,u_1,v_2)=-V_c(u_1,v_1,u_2,v_2)$ and $V_c(u_1,v_2,u_2,v_1)=-V_c(u_1,v_1,u_2,v_2)$. Thus, using \eqref{p:pr:equation:2}, we get that \eqref{eq:th:main:new1} implies
\begin{align}\label{eq:th:main:additive}
    0=\int_{p_1}^{q_1}\int_{p_2}^{q_2}\int_{p_1}^{q_1}\int_{p_2}^{q_2}H(u_1,v_1,u_2,v_2)dv_2 du_2 dv_1 du_1.
\end{align}
Using the multi-variable chain rule to differentiate with respect to $q_1$ and changing the order of integration, we get
\[
0=\int_{p_1}^{q_1}\int_{p_2}^{q_2}\int_{p_2}^{q_2}H(q_1,v_1,u_2,v_2)dv_2 dv_1 du_2 + \int_{p_1}^{q_1}\int_{p_2}^{q_2}\int_{p_2}^{q_2}H(u_1,v_1,q_1,v_2)dv_2  dv_1 du_1,
\]
and consequently, due to the symmetry of $H$ (in $(u_1,u_2)$), we have
\begin{equation}\label{eq:th:main:new2}
0=\int_{p_1}^{q_1}\int_{p_2}^{q_2}\int_{p_2}^{q_2}H(q_1,v_1,u_2,v_2)dv_2 dv_1 du_2 .
\end{equation}
Thus, differentiating \eqref{eq:th:main:new2} with respect to $p_1$ yields
\begin{equation*}
0=\int_{p_2}^{q_2}\int_{p_2}^{q_2}H(q_1,v_1,p_1,v_2)dv_2 dv_1.
\end{equation*}
Performing a similar operation again, i.e. differentiating with respect to $q_2$ and then $p_2$, we finally get
\begin{align}\label{eq:th:main:3}
H(q_1,q_2,p_1,p_2)=0.
\end{align}
Now,  using the strict monotonicity of $Q_X$ and $Q_Y$, we get $(Q_X(q_1)-Q_X(p_1))(Q_Y(q_2)-Q_Y(p_2))> 0$. Thus, directly from the definition of $H$, we get that \eqref{eq:th:main:3} implies
\begin{equation}\label{eq:th:main:new5}
c(q_1,q_2)c(p_1,p_2)-c(q_1,p_2)c(p_1,q_2)=0,
\end{equation}
for $0<p_1<q_1<1$ and $0<p_2<q_2<1$. This concludes the proof of~\eqref{eq:copula_dens} since for $q_1=p_1$ or $q_2=p_2$ the equality \eqref{eq:th:main:new5} is trivial, and the symmetry of $H$ allows us to easily extend \eqref{eq:th:main:3} to the full parameter space $q_1,q_2,p_1,p_2\in (0,1)$.

Second, let us show that~\eqref{eq:copula_dens} implies the independence of $X$ and $Y$. Using \eqref{eq:copula_dens} and recalling that $C$ could be seen as a distribution function of a random vector with marginals distributed uniformly on $[0,1]$,  for any $x,y\in [0,1]$ we get
\begin{align}\label{eq:th:main:copula}
C(x,y) &=\int_0^x\int_0^y c(u,v)dv du = \int_0^x\int_0^y\int_0^1\int_0^1 c(u_1,v_1)c(u_2,v_2)dv_2du_2dv_1  du_1  \nonumber \\
& = \int_0^x\int_0^y\int_0^1\int_0^1 c(u_1,v_2)c(u_2,v_1)dv_2du_2dv_1  du_1  \nonumber \\
&= \int_0^x\int_0^1 c(u_1,v_2)dv_2 du_1 \int_0^1\int_0^y c(u_2,v_1)dv_1du_2   \nonumber \\
& =C(x,1)C(1,y)  \nonumber \\
&=xy,
\end{align}
which shows that the copula of $(X,Y)$ is the product copula. Recalling~\eqref{eq:copula}, we get that $X$ and $Y$ are independent, which concludes the proof.
\end{proof}

\begin{remark}[Conditional Spearman's $\rho$ and independence]
From Theorem~\ref{th:main} one can easily deduce that random variables $X$ and $Y$ are independent if and only if conditional Spearman's $\rho$ coefficient on every quantile set is equal to zero. To prove this is it enough to observe that Spearman's $\rho$ is in fact Pearson's correlation applied to the copula function.
\end{remark}

\begin{remark}[Local linear independence implies independence]
By investigating the proof of Theorem~\ref{th:main} one can see that proving Equality~\eqref{eq:th:main:new5} is a key step in establishing independence. While in Theorem~\ref{th:main} we did not set any restriction on quantile split values $0<p_1<q_1<1$ and $0<p_1<q_1<1$, it is in fact sufficient to require that for any quantile point $(Q_X(p),Q_Y(q))$, $p,q\in (0,1)$, the quantile conditional correlations are null inside some neighbourhood of $(Q_X(p),Q_Y(q))$, e.g. for some $\epsilon>0$ and any $p_1,q_1\in (Q_X(p-\epsilon),Q_X(p+\epsilon))$ and $p_2,q_2\in (Q_Y(q-\epsilon),Q_Y(q+\epsilon))$. Indeed, this implies that~\eqref{eq:th:main:additive} is satisfied for any sufficiently small hypercubes $[p_1,q_1]\times [p_1,q_1]\times[p_2,q_2]\times [p_2,q_2]$ which can be combined to recover~\eqref{eq:th:main:additive} for any $0<p_1<q_1<1$ and $0<p_1<q_1<1$ and, consequently, get Equality~\eqref{eq:th:main:new5}.
This effectively shows than random variables $X$ and $Y$ are independent if and only if they are locally linearly independent.
\end{remark}

\begin{remark}[Tail-event dependence and spatial contagion]
From Theorem~\ref{th:main} we can see that to reject the (global) independence of $X$ and $Y$, it is enough to find a single quantile split $0<p_1<q_1<1$ and $0<p_2<q_2<1$ on which the conditional quantile correlation is not equal to zero. In signal processing or financial time-series modelling, it is natural to consider left tail events, e.g. when one or both of the values $q_1$ and $q_2$ are small. Such events could be linked to the presence of the so-called {\it spatial contagion} in which dependence increases in the presence of system turbulence. This might be used to construct statistical frameworks based on quantile tail-event analysis, see \cite{DurJaw2010}, \cite{JawPit2015}, and references therein.
\end{remark}

As we show now, Theorem~\ref{th:main} could be extended to the multivariate case. To get this extension, we use two alternative approaches. First, in Theorem~\ref{th:mutlivariate}, we consider a conditional correlation matrix. Second, in Theorem~\ref{th2}, we use linear combinations of margins.

Before we state the result, let us introduce some notation. Consider an $n$-dimensional random vector $X=(X_1,\ldots,X_n)$ and assume that it satisfies the assumptions analogous to the ones used in Theorem~\ref{th:main}, i.e. bijectiveness of the the marginal distribution functions, full support condition, absolute continuity of the joint distribution, and continuity of the copula density. Also, for any quantile splits $0<p_i<q_i<1$, $i=1, \ldots, n$, we define the quantile set corresponding to $(X_1,\ldots,X_n)$ by
\begin{equation}\label{eq:A_multidim}
A:=\bigcap_{i=1}^n\{Q_i(p_i)\leq X_i\leq Q_i(q_i)\},
\end{equation}
where $Q_i$ is the quantile function of $X_i$, $i=1, \ldots, n$. Finally, by $\Sigma_A$ we denote the associated conditional correlation matrix with the entries given by $\Sigma_A[i,j]:=[\cor_A(X_i,X_j)]$, $i,j=1, \ldots, n$; we also use $\textrm{I}_n$ to denote the $n\times n$ identity matrix.

\begin{theorem}\label{th:mutlivariate}
The $n$-dimensional random vector $X$ has (jointly) independent margins if and only if for any quantile splits $0<p_i<q_i<1$, $i=1, \ldots, n$, and the related set $A$, its conditional correlation matrix $\Sigma_A$ is equal to the identity matrix.
\end{theorem}

\begin{proof}
The argument is based on the proof of Theorem~\ref{th:main} and we provide only an outline. Also, for simplicity, we consider only $n=3$; the general case follows the same logic. As before, it is straightforward to check that the independence of margins imply diagonal conditional correlation matrix, so we focus on the reverse implication.

For simplicity and with a slight abuse of notation, we use $C$ and $c$ to denote the copula and the copula density corresponding to $(X_1,X_2,X_3)$, respectively. Let us assume that for any quantile splits $0<p_i<q_i<1$, $i=1, 2,3$, and the related set $A$ we get $\Sigma_A=\textrm{I}_3$ or equivalently
\begin{equation}\label{p:pr:equation:2_multi}
\cov_{A}[X_i,Y_j]=0, \quad i,j\in \{1,2,3\}, \, i\neq j.
\end{equation}
Next, as in~\eqref{eq:th:main:new1}, we get
\begin{align}
    \Cov_{A}[X_1,X_2]=\frac{1}{2\bP^2[A]}\int_{p_1}^{q_1}\int_{p_2}^{q_2}\int_{p_3}^{q_3}\int_{p_1}^{q_1}\int_{p_2}^{q_2}\int_{p_3}^{q_3}Q_X(u_1)Q_Y(v_1)V_c(u_1,v_1,w_1,u_2,v_2,w_2)dw_2dv_2 du_2 dw_1 dv_1 du_1,
\end{align}
where $V_c(u_1,v_1,w_1,u_2,v_2,w_2):=\left(c(u_1,v_1,w_1)c(u_2,v_2,w_2)-c(u_1,v_2,w_1)c(u_2,v_1,w_2)\right)$, $u_1,v_1,w_1,u_2,v_2,w_2\in (0,1)$. Also, setting $Q(u_1,v_1,u_2,v_2):=\left(Q_1(u_1)-Q_1(u_2)\right)\left(Q_2(v_1)-Q_2(v_2)\right)$, $u_1,u_2,v_1,v_2\in (0,1)$, and repeating the argument leading to~\eqref{eq:th:main:3}, we get
\[
0=\int_{p_3}^{q_3}\int_{p_3}^{q_3}Q(q_1,q_2,p_1,p_2)  V_c(q_1,q_2,w_1,p_1,p_2,w_2)dw_2 dw_1.
\]
Noting that $Q(q_1,q_2,p_1,p_2) >0$ and differentiating the iterated integral with respect to $q_3$ and $p_3$, for any $0<p_i<q_i<1$, $i=1, 2,3$, we get
\begin{align*}
    0& = V_c(q_1,q_2,q_3,p_1,p_2,p_3)+V_c(q_1,q_2,p_3,p_1,p_2,q_3)\\
    & = c(q_1,q_2,q_3)c(p_1,p_2,p_3)-c(q_1,p_2,q_3)c(p_1,q_2,p_3)+c(q_1,q_2,p_3)c(p_1,p_2,q_3)-c(q_1,p_2,p_3)c(p_1,q_2,q_3).
\end{align*}
In fact, as in the proof of Theorem~\ref{th:main}, we get that the formula is valid for any $p_i,q_i\in (0,1)$, $i=1,2,3$; see the discussion following~\eqref{eq:th:main:new5} for details. Using this observation, for any $x,y,z\in [0,1]$, as in~\eqref{eq:th:main:copula}, we get
\begin{align}\label{eq:th:multi_1}
    C(x,y,z)&=\int_0^x\int_0^y\int_0^z \int_0^1 \int_0^1 \int_0^1 c(u_1,v_1,w_1)c(u_2,v_2,w_2)dw_2dv_2 du_2 dw_1 dv_1 du_1 \nonumber\\
    & = \int_0^x\int_0^y\int_0^z \int_0^1 \int_0^1 \int_0^1  c(u_1,v_2,w_1)c(u_2,v_1,w_2)dw_2dv_2 du_2 dw_1 dv_1 du_1 \nonumber\\
    & \phantom{=} - \int_0^x\int_0^y\int_0^z \int_0^1 \int_0^1 \int_0^1  c(u_1,v_1,w_2)c(u_2,v_2,w_1)dw_2dv_2 du_2 dw_1 dv_1 du_1 \nonumber\\
    & \phantom{=} + \int_0^x\int_0^y\int_0^z \int_0^1 \int_0^1 \int_0^1  c(u_2,v_1,w_2)c(u_1,v_2,w_1)dw_2dv_2 du_2 dw_1 dv_1 du_1 \nonumber\\
    & = C(x,1,z)C(1,y,1)-C(x,y,1)C(1,1,z)+C(1,y,z)C(x,1,1) \nonumber\\
    & = C(x,1,z)y-C(x,y,1)z+C(1,y,z)x.
\end{align}
In particular, setting $z=1$, we get $C(x,y,1)=xy$, $x,z\in [0,1]$. 
Using the same argument applied to $\Cov_{A}[X_1,X_3]$ and $\Cov_{A}[X_2,X_3]$, we also get $C(x,1,z)=xz$ and $C(1,y,z)=yz$, $x,y,z\in [0,1]$.
Consequently, from~\eqref{eq:th:multi_1}, we get $C(x,y,z)=xyz$, $x,y,z\in [0,1]$, which concludes the proof.
\end{proof}

The next generalisation of Theorem~\ref{th:main} is based on linear combinations. For simplicity, given an $n$-dimentional random vector $X$ and an $m$-dimensional random vector $Y$, we pre-assume that for any $\alpha \in \mathbb{R}^n\setminus \{0\}$ and $\beta \in \mathbb{R}^m\setminus \{0\}$, the random vector ($\langle X, \alpha\rangle,\langle Y, \beta\rangle)$, where $\langle \cdot, \cdot \rangle$ denotes the standard Euclidean inner product, satisfy our usual assumptions, i.e. bijectiveness of the the marginal distribution functions, full support condition, absolute continuity of the joint distribution, and continuity of the copula density.

\begin{theorem}\label{th2}
Let $X$ and $Y$ be $n$-dimensional and $m$-dimensional random vectors, respectively. Then, $X$ and $Y$ are independent if and only if for any $\alpha \in \mathbb{R}^n\setminus \{0\}$ and $\beta \in \mathbb{R}^m\setminus \{0\}$, the random variables $\langle X, \alpha\rangle$ and  $\langle Y, \beta\rangle$ are conditionally uncorrelated, i.e. for any  quantile splits $0<p_1<q_1<1$ and $0<p_2<q_2<1$ and the related set $A$, defined for  $(\langle X, \alpha\rangle,\langle Y, \beta\rangle)$, we get
\begin{equation}\label{eq:th2:covariance}
\cor_{A}\left[\langle X, \alpha\rangle,\langle Y, \beta\rangle\right]=0.
\end{equation}
\end{theorem}
\begin{proof}
As in the proof of Theorem~\ref{th:main}, we focus on the argument that~\eqref{eq:th2:covariance} implies independence; the reverse implication is standard. Note that~\eqref{eq:th2:covariance} combined with Theorem~\ref{th:main} implies that the random variables $\langle X, \alpha\rangle$ and $\langle Y, \beta\rangle$ are independent for any $\alpha \in \mathbb{R}^n$ and $\beta \in \mathbb{R}^m$. In particular, we get
\begin{equation}\label{eq:th2:char_funct}
\phi_{(X,Y)}(\alpha,\beta)=\phi_X(\alpha)\phi_Y(\beta), \quad \alpha\in \mathbb{R}^n, \, \beta \in \mathbb{R}^m,  
\end{equation}
where $\phi_Z(t):=\mathbb{E}[\exp\left(i\langle Z,t\rangle\right)]$, $t\in \mathbb{R}^d$, denotes the characteristic function of an arbitrary $d$-dimensional random vector $Z$. Combining~\eqref{eq:th2:char_funct} with Theorem 4, Section II.12, in~\cite{Shi1996} we conclude the proof.
\end{proof}

Theorem~\ref{th2} can be used to get another characterisation of random vector margins independence based on a recursive scheme. For an $n$-dimensional vector $a:=(a_1, \ldots, a_n)$ and $k=1, \ldots, n$, let $a^{1:k}$ denote its subvector $a^{1:k}:=(a_1, \ldots, a_k)$. Again, for simplicity, for $n$-dimensional random vector $X$, we assume that for any $k=1, \ldots, n-1$ and $\alpha^{k}\in \mathbb{R}^k\setminus\{0\}$, the random variables $X_{k+1}$ and $\langle X^{1:{k}}, \alpha^k\rangle$ satisfy our standard assumptions.

\begin{theorem}\label{th:TS}
Let $X$ be an $n$-dimensional random vector. Assume that for any $k=1, \ldots, n-1$ and $\alpha^{k}\in \mathbb{R}^k\setminus\{0\}$, the random variables $X_{k+1}$ and $\langle X^{1:{k}}, \alpha^k\rangle$ are conditionally uncorrelated. Then, the margins of $X$ are (jointly) independent. 
\end{theorem}
\begin{proof}
Using Theorem~\ref{th2}, we get that, for any $\alpha^{n-1}\in \mathbb{R}^{n-1}\setminus\{0\}$, the random variables $X_n$ and $\langle X^{1:{(n-1)}}, \alpha^{n-1}\rangle$ are independent. Hence, the characteristic function of the random vector $X$ satisfies $\phi_X(\alpha)=\phi_{X_n}(\alpha_n)\phi_{X^{1:(n-1)}}(\alpha^{1:(n-1)})$, $\alpha \in \mathbb{R}^n$. In fact, inductively, we get that the characteristic funcion of $X$ factorises into the product of the characteristic functions of the margins, i.e.
\[
\phi_X(\alpha)=\prod_{k=1}^n \phi_{X_i}(\alpha_i), \quad \alpha \in \mathbb{R}^n. 
\]
Using Theorem 4, Section II.12, in~\cite{Shi1996} we conclude the proof.
\end{proof}
\begin{remark}
It should be noted that Theorem~\ref{th:TS} could be used to get a characterisation of infinite series of random variables, e.g. $(X_t)_{t\in \bN}$. Indeed, it is enough to recall that the independence of the family $(X_t)_{t\in \bN}$ means independence of any finite subfamily of random variables.
\end{remark}


\section{Examples}\label{S:examp}
We conclude this note with two simple examples which illustrate Theorem~\ref{th:main}. Example~\ref{ex:1} refers to a classic example of uncorrelated random variables which are not independent, while Example~\ref{ex:2} shows how one can localise the auto-correlation analysis in signal processing to study time-series independence. In particular, it should be noted that our framework allows to consider all classes of processes including the ones with heavy-tails, for which the unconditional auto-correlation function might not exist.

\begin{example}[Uncorrelated normal random variables that are not independent]\label{ex:1}
Let us consider a classic example of two uncorrelated normal random variables that are not independent, see \cite{Bro1986}.  Let $X$ be a standard normal random variable and $W$ be a Rademacher random variable. Let us assume that $X$ and $W$ are independent and set $Y:=WX$. Using a classic argument one may show that: (1) $Y$ is standard normal random variable; (2) $X$ and $Y$ are (unconditionally) uncorrelated; (3) $X$ and $Y$ are not independent. To illustrate how Theorem~\ref{th:main} works, we provide an explicit formula for the quantile conditional covariance of $X$ and $Y$.

To ease the notation, we denote by $\Phi$ and $\phi$ the standard normal cumulative distribution function and probability density function, respectively. Also, we fix some quantile splits $0<p_1<q_1<1$ and $0<p_2<q_2<1$ and define
$l:=\max(\Phi^{-1}(p_1), \Phi^{-1}(p_2))1_{\{q_1>p_2\}}1_{\{q_2>p_1\}}$, $r:=\min(\Phi^{-1}(q_1), \Phi^{-1}(q_2))1_{\{q_1>p_2\}}1_{\{q_2>p_1\}}$, $\hat{l}:=\max(\Phi^{-1}(p_1), \Phi^{-1}(1-q_2))1_{\{q_1>1-q_2\}}1_{\{1-p_2>p_1\}}$, and $\hat{r}:=\min(\Phi^{-1}(q_1), \Phi^{-1}(1-p_2))1_{\{q_1>1-q_2\}}1_{\{1-p_2>p_1\}}$;
note that $l$ and $r$ are simply the left-most and the right-most points of the set $[\Phi^{-1}(p_1),\Phi^{-1}(q_1)]\cap[\Phi^{-1}(p_2),\Phi^{-1}(q_2)]$, respectively, provided that the intersection is non-empty; a similar interpretation holds for $\hat{l}$ and $\hat{r}$. With this notation, it is easy to check that
\begin{align}\label{eq:covariance_normal}
    \Cov_{A}[X,Y]&=\frac{l\phi(l)-r\phi(r)+\Phi(r)-\Phi(l)}{\Phi(r)-\Phi(l)+\Phi(\hat{r})-\Phi(\hat{l})} - \frac{\hat{l}\phi(\hat{l})-\hat{r}\phi(\hat{r})+\Phi(\hat{r})-\Phi(\hat{l})}{\Phi(r)-\Phi(l)+\Phi(\hat{r})-\Phi(\hat{l})} -\frac{(\phi(l)-\phi(r))^2-(\phi(\hat{l})-\phi(\hat{r}))^2}{(\Phi(r)-\Phi(l)+\Phi(\hat{r})-\Phi(\hat{l}))^2}.
\end{align}
In Figure~\ref{F:contour}, we present the values of $\Cov_{A}[X,Y]$ for exemplary quantile splits.
\begin{figure}[htp!]
\begin{center}
\includegraphics[width=0.25\textwidth]{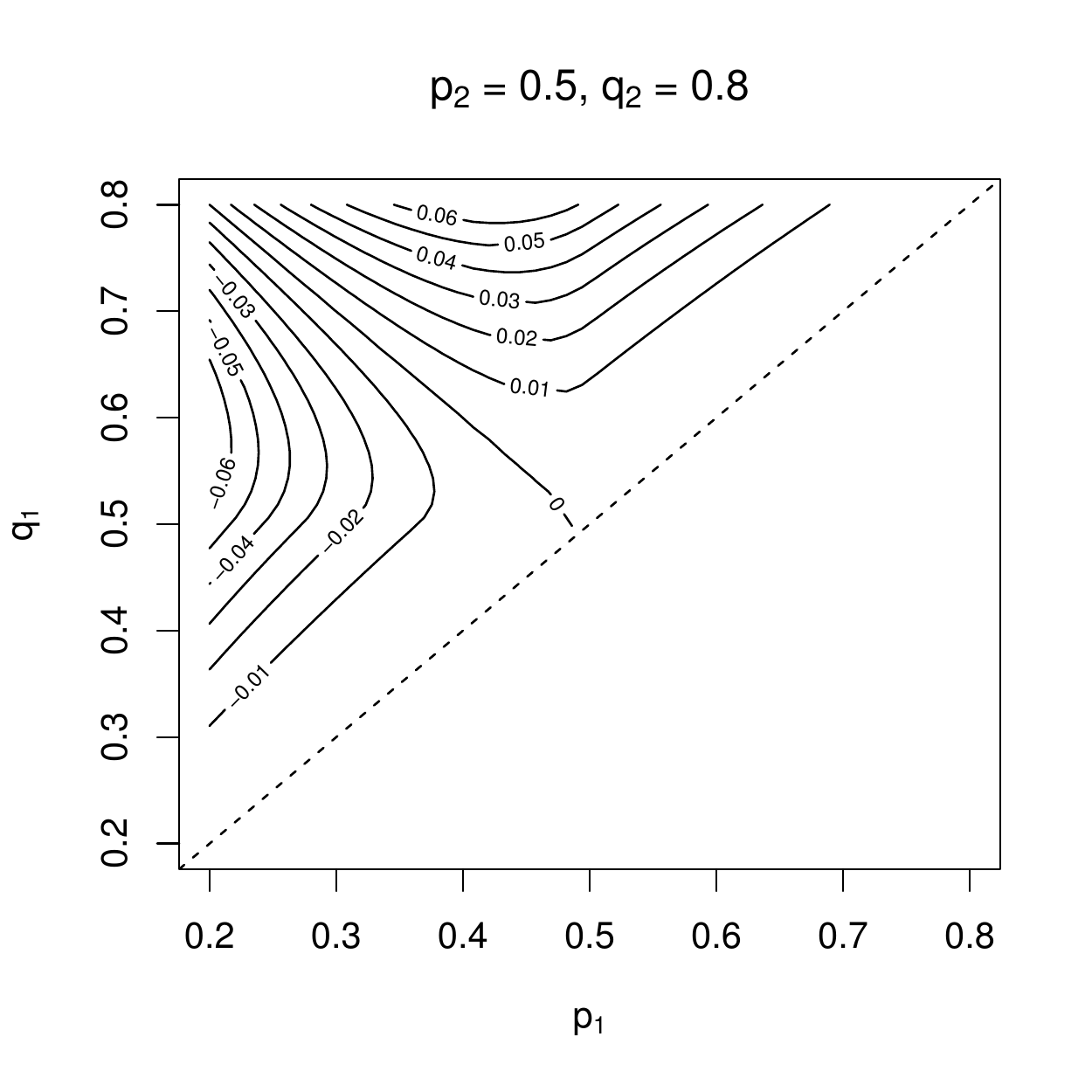}
\includegraphics[width=0.25\textwidth]{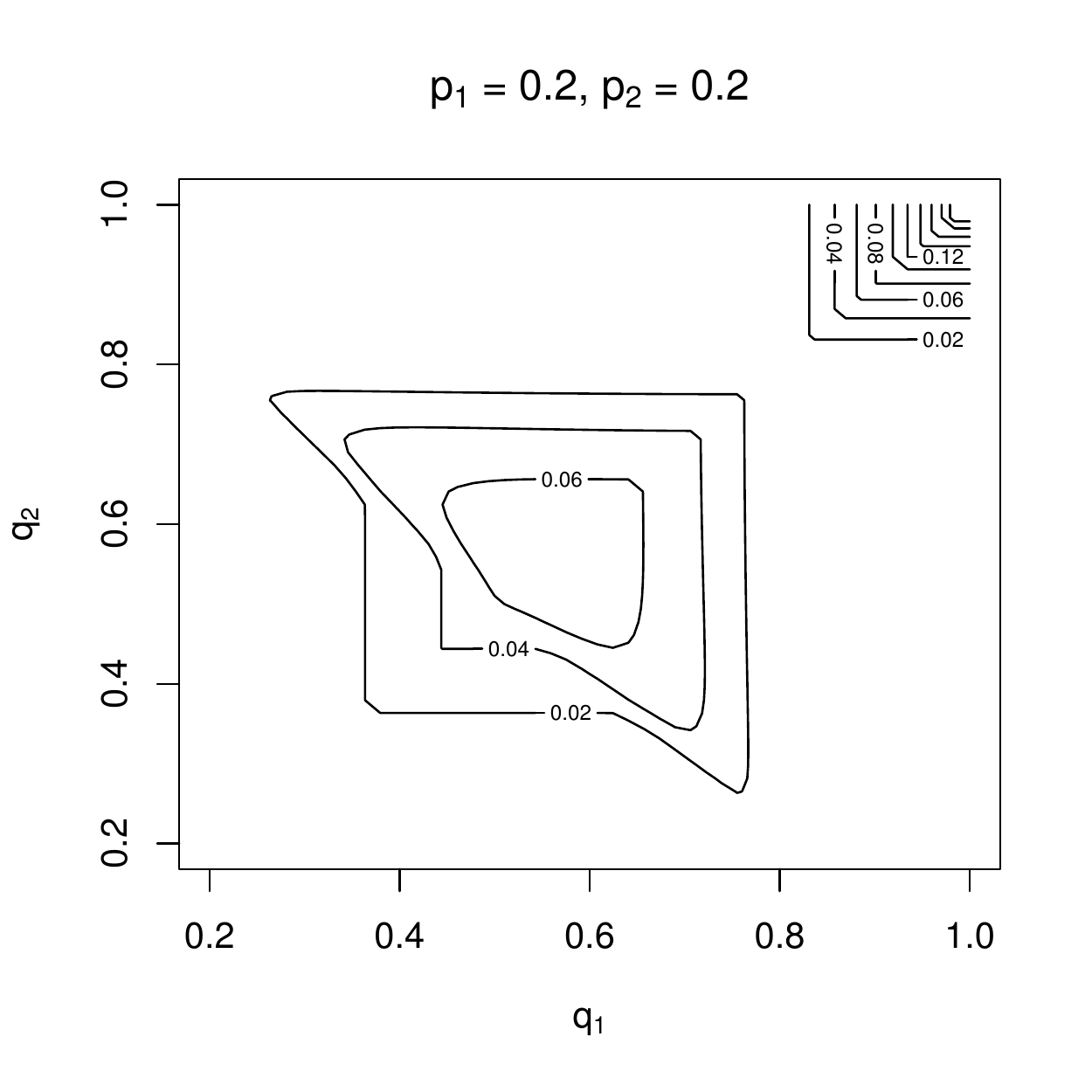}
\end{center}
\caption{Contour plots of $\Cov_A[X,Y]$ for various choices of quantile splits and the corresponding set $A$ given by~\eqref{eq:A}. The left panel shows $0.2<p_1<q_1<0.8$ and $p_2=0.5$, $q_2=0.8$ while the right panel shows $p_1=p_2=0.2$ and $q_1, q_2\in (0.2, 1.0)$. The results are based on Equation~\eqref{eq:covariance_normal}.}\label{F:contour}
\end{figure}
From Figure~\ref{F:contour} we see that the quantile conditional covariances may detect dependence between random variables, as described in Theorem~\ref{th:main}. In particular, for $p_1=p_2=0.5$ and $q_1=q_2=0.8$, we get $ \Cov_{A}[X,Y]\approx 0.0573$, which directly proves that $X$ and $Y$ are not independent. 
\end{example}

\FloatBarrier

\begin{example}[Auto-correlation analysis]\label{ex:2}
The analysis of auto-correlation is a common and standard technique used in time series analysis and signal processing to detect serial dependence in time-series, see e.g. \cite{Ham1994,BroDav2016} and references therein. In particular, it is often used to verify lack of trend or volatility clustering in financial data, see e.g. \cite{FamFre1988,Con2001,JiaSaaXia2016}. Given a time-series sample $(x_t)_{t=1}^{n}$, the empirical auto-correlation of lag $k$ is typically computed by estimating the (unconditional) correlation between the sub-samples $(x_{t})_{t=k+1}^{n}$ and $(x_{t})_{t=1}^{n-k}$. In this simple example, we use market data to show how the information about conditional correlation could be used to refine standard auto-correlation function (ACF) analysis. For simplicity, we decided to take one exemplary stock market data. Namely, we consider weekly (adjusted price) returns of AAPL stock in the period 01/08/2016 -- 01/08/2022, the data is illustrated in Figure~\ref{F:1}.

\begin{figure}[htp!]
\begin{center}
\includegraphics[width=0.4\textwidth]{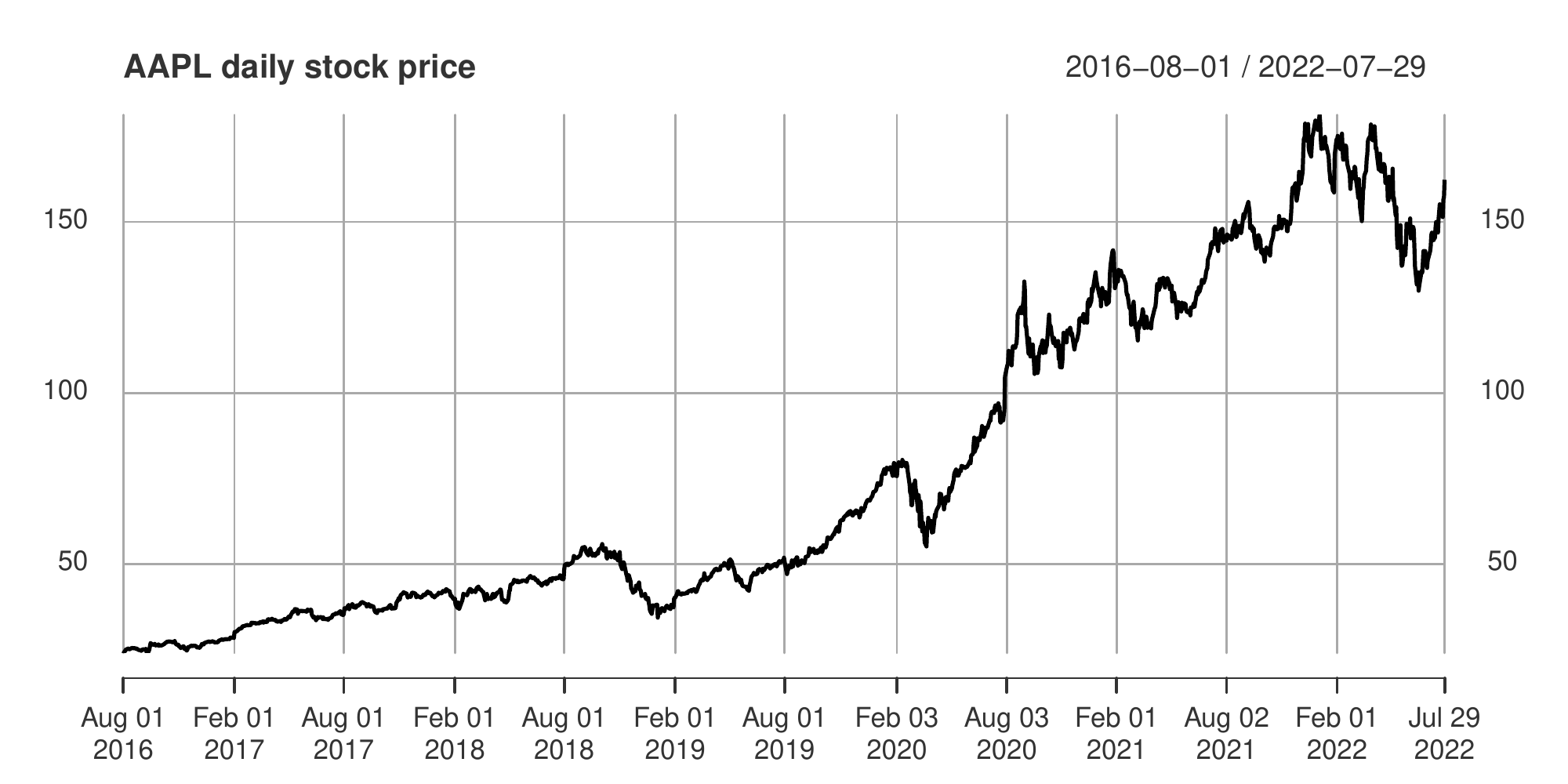}
\includegraphics[width=0.4\textwidth]{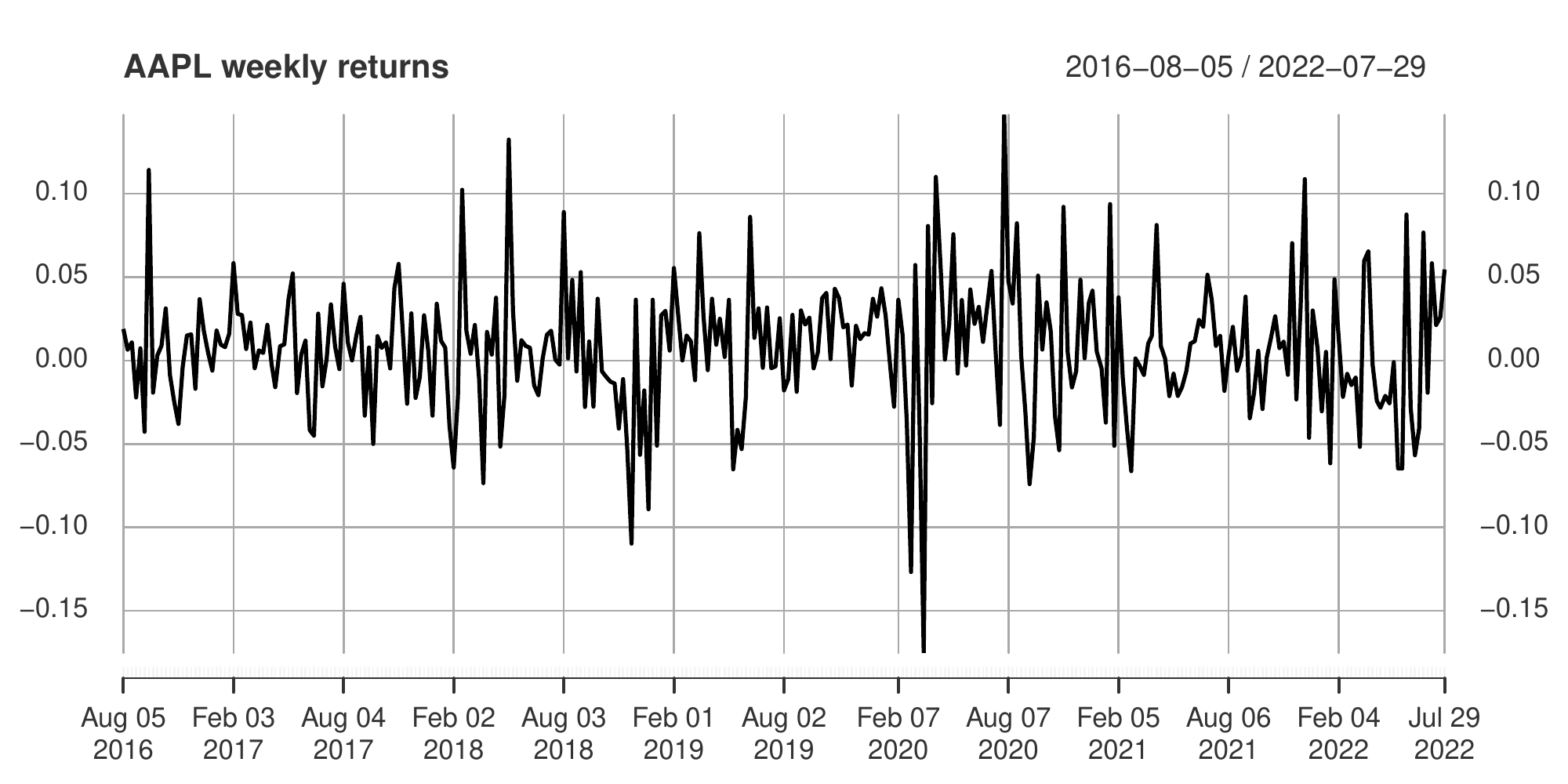}
\end{center}
\caption{AAPL stock price daily quotes and weekly log-returns in the period 01/08/2016 -- 01/08/2022.}\label{F:1}
\end{figure}

In Figure~\ref{F:2}, we present the classical auto-correlation function (ACF) plots for log-returns, absolute values of log-returns, as well as squared log-returns. 
\begin{figure}[htp!]
\begin{center}
\includegraphics[width=0.30\textwidth]{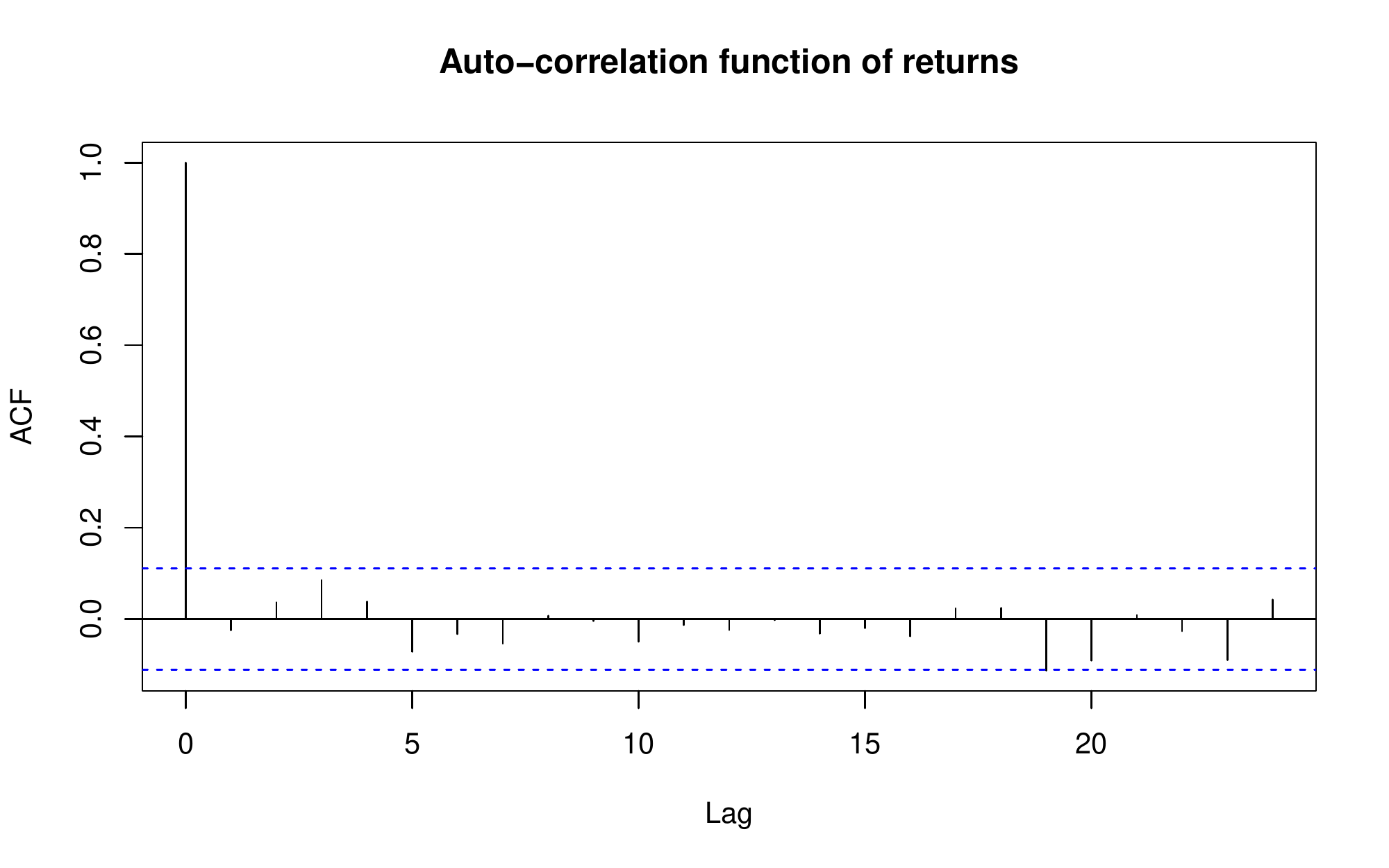}
\includegraphics[width=0.30\textwidth]{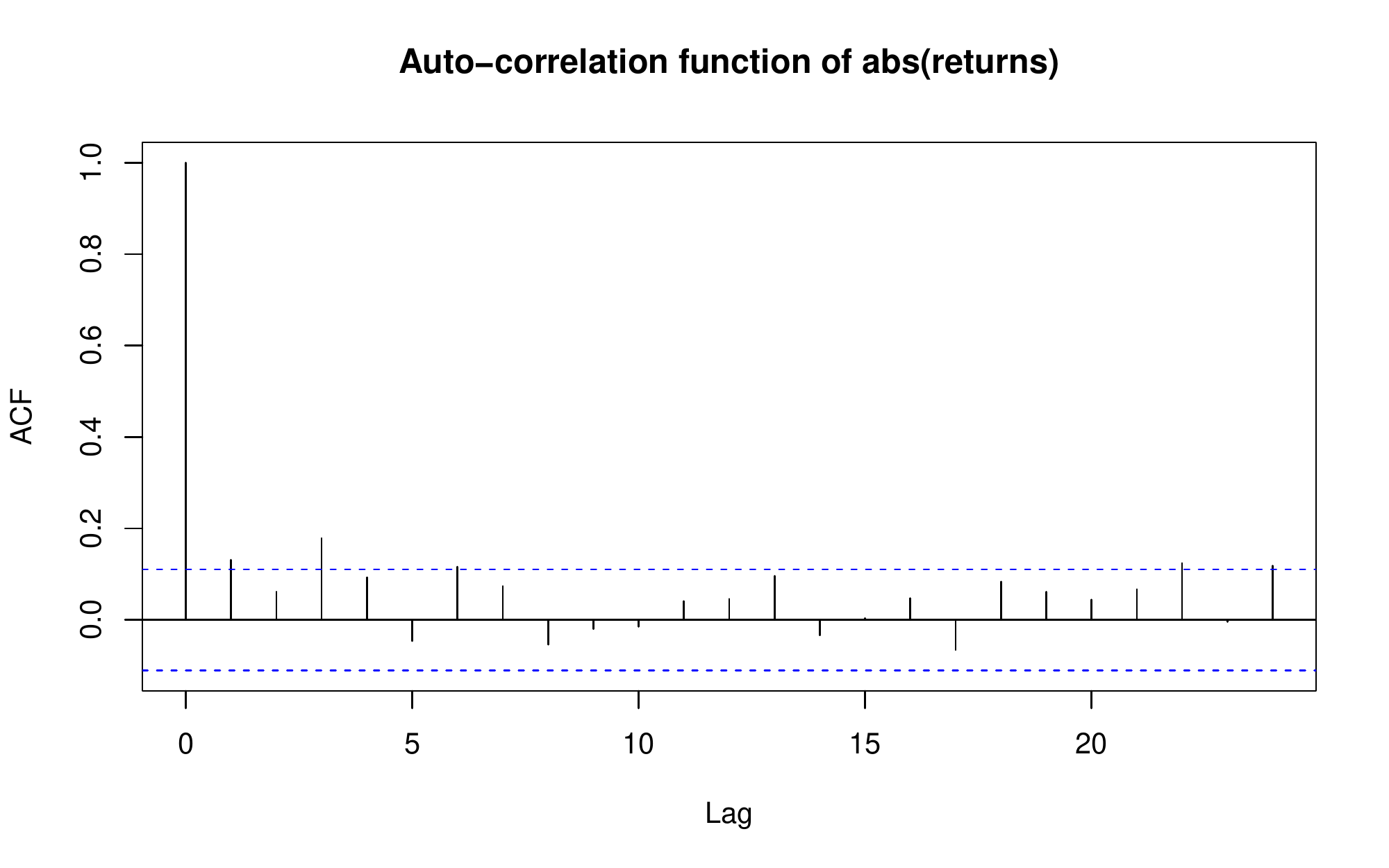}
\includegraphics[width=0.30\textwidth]{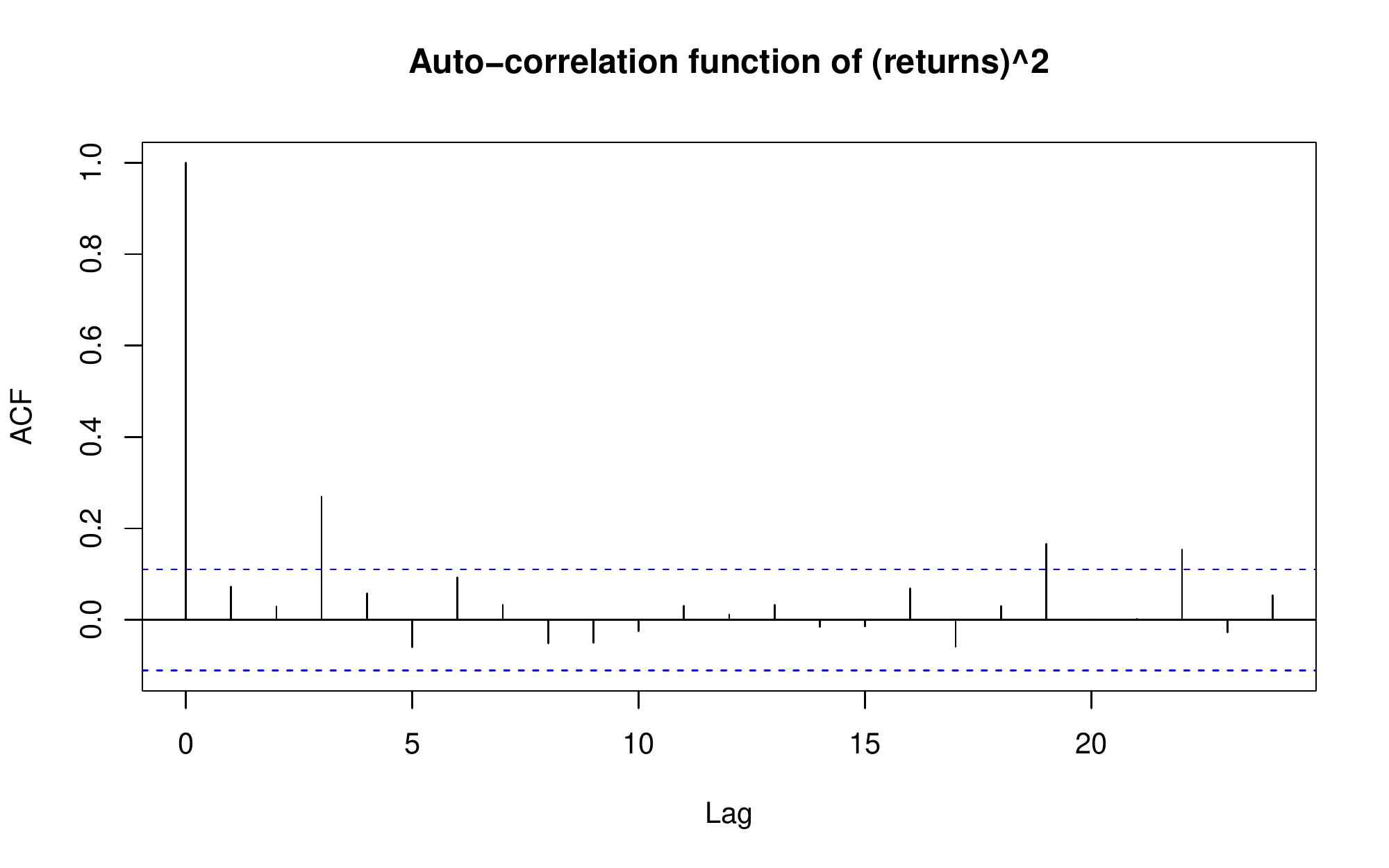}
\end{center}
\caption{Auto-correlation function (ACF) plots for AAPL stock price weekly log-returns in the period 01/08/2016 -- 01/08/2022. The left exhibit shows ACF applied to log-returns, the middle exhibit shows ACF applied to absolute values of log-returns, while the right exhibit shows ACF applied to squared log-returns.}\label{F:2}
\end{figure}
While, the first plot is often used for generic independence check (lack of trend), the last two might be used to investigate the so-called volatility clustering effect, see~\cite{Con2001}. Although from Figure~\ref{F:1} one can deduce that the data is not i.i.d., the ACF functions do not detect any major problem, especially for lag $k=1$.

Let us now focus on lag $k=1$ and check if we can refine the ACF analysis using conditional auto-correlation function rather than the unconditional one. From Theorem~\ref{th:main} we know that to check independence between consecutive observations, we can calculate their local auto-correlation  on a quantile set, i.e. empirical correlation for conditioned samples $(x_{t})_{t=2}^{n}$ and $(x_{t})_{t=1}^{n-1}$. Namely, let us consider the quantile split $p_1=p_2=0.01$ and $q_1=q_2=0.7$, and compute the conditional correlation on the corresponding set $A$ given by~\eqref{eq:A}. In Figure~\ref{F:3}, we present the lagged sample plot where the values of $x_t$ are confronted with the values of $x_{t-1}$; the red data-points indicate conditional sub-sample on which correlation is computed. The estimated value of the conditional correlation is equal to 0.31, which indicate that the time-series observations are not independent. To sanity check if this claim is statistically significant, we performed a simple normal distribution based Monte Carlo exercise. Namely, we picked $M=100\,000$ strong Monte Carlo samples of size $n=312$, i.e. size equal to the size of the original sample, from independent normal distributions. For each run, we computed the conditional correlation for the same lag and the same sample quantile set. The 0.1\% upper quantile of the obtained MC density is equal to 0.25, which shows that the initial sample empirical correlation 0.31 is (statistically) significantly different  from zero.

\begin{figure}[htp!]
\begin{center}
\includegraphics[width=0.3\textwidth]{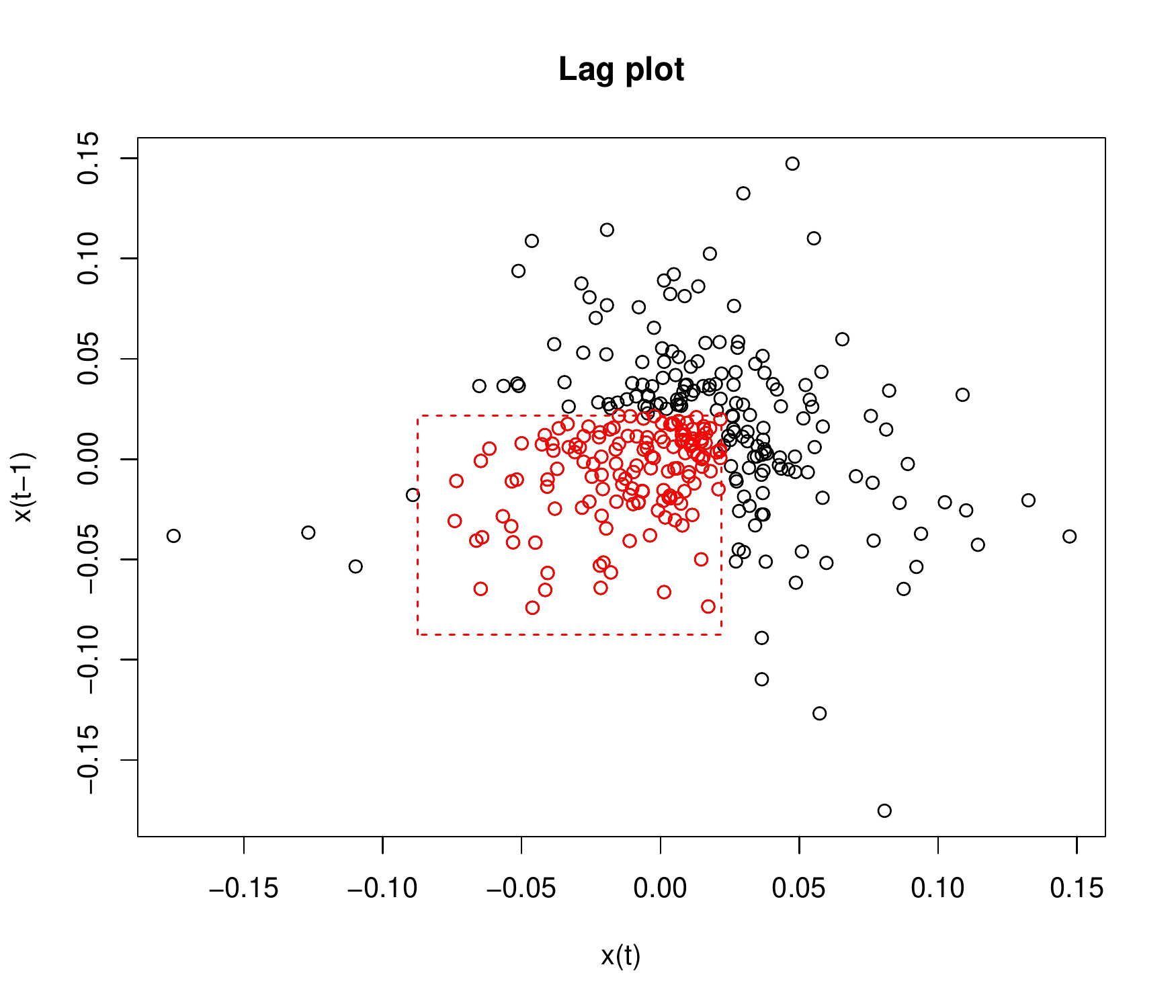}
\end{center}
\caption{Lag plot ($k=1$) for weekly APPL log-returns. The red region indicates the conditional set $A$ on which we computed (empirical) conditional correlation which is equal to 0.31. This indicates that the consequent observations are not independent.}\label{F:3}
\end{figure}
We want to note that while the analysis performed in this example is simplistic, and the obtained statistical significance is based on strong normality assumptions, the obtained results indicate that the conditional version of the auto-correlation function might be useful in time-series analysis. Also, note that such analysis could be applied even to heavy tailed data for which the unconditional auto-correlation might not exists; this topics are left for the future research.
\end{example}


\section*{Acknowledgements}
Marcin Pitera acknowledges support from the National Science Centre, Poland, via project 2020/37/B/HS4/00120. Part of the work of Damian Jelito was funded by the Priority Research Area Digiworld under the program Excellence Initiative – Research University at the Jagiellonian University in Krak\'{o}w. 
\newpage
\begin{footnotesize}
\bibliographystyle{agsm}
\bibliography{mybibliography}
\end{footnotesize}


\end{document}